\documentclass[preprint,authoryear]{elsarticle}
\pdfoutput=1
\usepackage{lineno}
\modulolinenumbers[5]

\journal{arXiv}

\usepackage{a4wide}
\usepackage{amsmath,amsfonts,amssymb,mathtools,amsthm,mathrsfs,bbm}
\usepackage{dsfont}
\usepackage{graphicx}
\usepackage{xcolor}
\usepackage{tikz}
\usepackage{enumerate, paralist}
\usepackage{bbm}
\usepackage{microtype}
\usepackage[mathscr]{euscript}
\usepackage{graphicx}
\usepackage[caption = false]{subfig}
\usepackage[percent]{overpic}

\newtheorem{proposition}{Proposition}[section]
\newtheorem{corollary}[proposition]{Corollary}

\newtheorem{theorem}{Theorem}
\newtheorem{lemma}[proposition]{Lemma}
\theoremstyle{definition}
\newtheorem{definition}[proposition]{Definition}
\newtheorem{remark}[proposition]{Remark}
\DeclareMathAlphabet{\mathpzc}{OT1}{pzc}{m}{it}
\numberwithin{equation}{section}
\numberwithin{figure}{section}

\newcommand{\ehalb}[1][]{e^{-#1\rho T}}
\newcommand{\etwohalb}[1][]{e^{-#1\rho T'}}
\newcommand{\eonetwohalb}[1][]{e^{-#1\rho (T+T')}}

\biboptions{authoryear}

\usepackage[bookmarks=false]{hyperref}

\hypersetup{
    hidelinks,
}

\begin{document}

\begin{frontmatter}

\title{The independent loss model with ordered insertions for the evolution of
  CRISPR spacers}

\author[mymainaddress]{F. Baumdicker\corref{mycorrespondingauthor}}
\cortext[mycorrespondingauthor]{Corresponding author}
\ead{franz.baumdicker@gmx.de}
\author[mymainaddress]{A. M. I. Huebner}
\author[mymainaddress]{P. Pfaffelhuber}

\address[mymainaddress]{Department of Mathematical Stochastics,
  Albert-Ludwigs-University of Freiburg}

\begin{abstract}
  Today, the CRISPR (clustered regularly interspaced short palindromic
  repeats) region within bacterial and archaeal genomes is known to
  encode an adaptive immune system. We rely on previous results on the
  evolution of the CRISPR arrays, which led to the ordered independent
  loss model, introduced by Kupczok and Bollback (2013). When focusing
  on the spacers (between the repeats), new elements enter a
  CRISPR array at rate $\theta$ at the leader end of the array,
  while all spacers present are lost at rate $\rho$ along the
  phylogeny relating the sample. Within this model, we compute the
  distribution of distances of spacers which are present in all arrays
  in samples of size $n=2$ and $n=3$. We use these results to
  estimate the loss rate $\rho$ from spacer array data.
\end{abstract}


\end{frontmatter}


\section{Introduction}

The CRISPR Cas system is a widespread microbial adaptive defense
mechanism against viruses and plasmids
\citep{Marraffini2015,Rath2015}, that likely originated in archaea and
spread to bacteria via horizontal transfer \citep{Makarova2011origin}.
The Clustered Regulary Interspaced Short Palindromic Repeats (CRISPR)
have already been described in 1987 by \cite{Ishino1987}.  Later it
turned out that the unique sequences between these repeats, so called
spacers, are of foreign origin \citep{Bolotin2005} and serve as an
immunological memory passed to the offspring.  New spacers are
acquired and inserted at the leader end of the array
\citep{Barrangou2007}, such that the order of spacers represents the
chronological infection history of the bacterial population. Together
with CRISPR associated (cas) genes these spacers can provide
resistance against phages and plasmids by targeting molecular scissors
to the corresponding sequences in the invading DNA
\citep{Barrangou2007}.

The most prominent cas gene is Cas9, which recently led to a
revolution in genome engineering \citep{Doudna2014,Hsu2014}.  With the
CRISPR-Cas9 enzyme mechanism an uncomplicated and cheap technology to
alter the genome of potentially any organism is now available.  Due to
the precise targeting via engineered spacer sequences this system is
speeding up the pace of research and gives rise to applications with
incredible impact and opportunities in a variety of fields
\citep{Doudna2014}.
As just one among many examples, the concept of gene drive
\citep{Burt2003} in combination with the precision of CRISPR-Cas9 may
enable us to alter the genetics of entire populations
\citep{Esvelt2014,Oye2014}.

Here we will focus on the evolution of natural CRISPR-Cas systems and
their spacer arrays in microbial genomes.  CRISPR systems have been
classified into different types and subtypes, with different sets of
accompanying cas genes
\citep{Makarova2011classification,Makarova2015}.  A single genome can
contain different types of CRISPR and the rates at which new spacers
are inserted and old spacers are lost vary between the systems
\citep{Horvath2008}.
This suggests that different types may have different evolutionary
dynamics and functions beyond defense, e.g. in regulation of gene
expression \citep{Westra2014}.  As another example,
\cite{Lopez-Sanchez2012} suggest that CRISPR may control the
diversity of mobile genetic elements in \emph{Streptococcus
  agalactiae}.

So far we just got a glimpse of the ecological and evolutionary impact
of CRISPR cas systems.  In particular, the benefit of possessing a
CRISPR system and the parameter regime where they are maintained
\citep{Levin2010,Weinberger2012}, as well as the coevolutionary
dynamics of bacteria containing CRISPR loci and phages
\citep{Koskella2014,Han2017} have been considered.  However many
ecological and evolutionary aspects of CRISPR cas systems, as the
frequent horizontal transfer of the whole system, are still not
understood \citep{Rath2015}.
Not only the evolution of the whole CRISPR system but the evolving and
adapting spacer array itself is of interest.  The spacer array
represents snippets of previous phage/plasmid exposure that can help
to disentangle the interplay between bacterial and viral populations
\citep{Childs2014,Sun2015}.  Since resistance is inherited by the
offspring via the acquired spacer sequences, at least some CRISPR
systems blur the distinction between Darwinian and Lamarckian modes of
evolution \citep{Koonin2016}.
Modeling the evolution of spacer arrays will help to identify
differences between CRISPR types and interpret the observed pattern of
spacer insertion and deletion.
 
In 2013, \citeauthor{Kupczok2013} introduced probabilistic models for
the evolution of CRISPR spacer arrays.  They concluded that a model with ordered spacer insertions
and independent losses best describes the dynamics of spacer
array evolution.  In this model unique new spacers are inserted at the
leader end and each spacer gets lost independently at a constant rate;
see also Definition \ref{def:oiloss} below. Later, the model has been
extended to conclude that there is no evidence for frequent
recombination within the spacer arrays \citep{Kupczok2015}.  In
\cite{Kupczok2013} the constructed estimators for spacer insertion and
deletion rates assume that no phylogenetic information is available.
In contrast, in this paper we assume that the genealogy is known or has been
reconstructed adequately, e.g.\ based on the cas genes in front of the
spacer array. Given such a genealogy we look at distances between
equal spacers, i.e.\ spacers that appear in more than one array.
\cite{Kupczok2013} also considered an unordered independent loss
model, where the order of spacers is irrelevant. In pangenome analysis
\citep{Mira2010,Vernikos2014}, this model is known as the infinitely
many genes model \citep{BaumdickerHessPfaffelhuber2010} and has led to
methods to jointly estimate gene gain and loss rates based on the
frequency of genes in the sample
\citep{BaumdickerPfaffelhuberHess2012}. These methods can directly be
applied in our setting to jointly infer the rates of spacer insertion
and deletion from (unordered) spacer frequencies.  Here we compute the
distribution of distances of (ordered) spacers, which are present in
all arrays in samples of size $n=2$ and $n=3$.  We show that including
the order of spacer arrays by looking at equal spacer distances in a
sample of arrays allows to decouple estimation of spacer insertion and
spacer loss rate.

The paper is organized as follows: In Section~\ref{S:2}, we introduce
the ordered independent loss model. In Section~\ref{S:3}, we compute
the distribution of equal spacer distances of samples of size $n=2$
and $n=3$. As a by-product, we find sufficient statistics useful for
the estimation of the loss rate. These are employed in
Section~\ref{S:4} where we give maximum likelihood estimators and
perform simulations in order to show their accuracy.

\section{Model}
\label{S:2}
The following model for the evolution of spacers in the CRISPR-system
is based on work by \cite{Kupczok2013}. Since this model was called
the ordered model with independent losses, we follow this terminology here.

\begin{definition}[The ordered independent loss model along a single line\label{def:oiloss}]
  Let $\mathcal S = (S_t)_{t\geq 0}$ be a Markov jump process with
  state space $E = [0,1]^{\mathbb N}$ and the following dynamics: If
  $S(t) = s = (s_1, s_2,\dots)$, it jumps to
  \begin{align*}
    (U, s_1, s_2,\dots) &\text{ at rate  } \theta \text{ for some independent, uniform $U\in[0,1]$},\\
    (s_1,\dots,s_{i-1}, s_{i+1},s_{i+2},\dots) & \text{ at rate } \rho \text{ for each $i=1,2,\dots$}
  \end{align*}
  The first type is called a {\em gain}-event (or insertion event), while the latter is called a {\em loss}-event (or deletion event).
  We refer to $S$ as the independent loss model with ordered gains, or as the \emph{ordered independent loss model}, for short. 
  
\end{definition}

\noindent
A simple result is the following, which is clear from the definition
of $\mathcal S$.

\begin{lemma}[Equilibrium of $S$\label{l:seq}]
  Let $(U_1, U_2,\dots)\in E$ be a sequence of independent,
  $U([0,1])$-distributed random variables. Then, the distribution
  $\pi$ of $(U_1, U_2,\dots)$ is an equilibrium of $\mathcal S$ from
  Definition~\ref{def:oiloss}.
\end{lemma}

\begin{remark}[More equilibria\label{rem:more}]
  We note that $\pi$ is not the only equilibrium of the process. For
  example, let $N \sim \text{Poi}(\theta/\rho)$ and, conditioned on
  $N$, let $(V_1,\dots,V_N)$ be independent and $U([0,1])$-distributed,
  and $V_{N+1} = V_{N+2}=\cdots =1$. Then, the distribution of
  $(V_1,V_2,\dots)$ is stationary for $S$ as well.\\
  Indeed, let $N_t:=\sup\{n: U_1,\dots,U_n\notin S(0)\}$. Then,
  $\mathcal N=(N_t)_{t\geq 0}$ is a death-immigration process with immigration
  rate $\theta$ and, if it is in state $n$, death rate
  $n\rho$. For this process, it is well-known that
  $N_t\xRightarrow{t\to\infty}N$; hence, necessarily, $V_1,\dots,V_N$ are
  independent and uniform draws from $[0,1]$ for a stationary
  distribution. However, $V_{N+1},V_{N+2},\dots$ are states also found
  in $S(0)$, so if $S(0)=(1,1,\dots)$, then $S_{N_t+k}(t)=1$ for all $t$
  and $k$ by construction.
\end{remark}

\sloppy In order to formulate our results, we require the ordered
independent loss model not only along a single line, but also along an
ultra-metric tree. For this, we introduce some notation.

\begin{remark}[Ultrametric trees; tree-indexed processes]
\label{rem:notUltra}
  \begin{asparaenum}
  \item Recall that a tree $\mathbbm T$ with leaves
    $\mathbbm L\subseteq \mathbbm T$ is called ultra-metric if
    $d(\ell_1,\ell_3)\leq \max(d(\ell_1, \ell_2), d(\ell_2,\ell_3))$
    for all $\ell_1, \ell_2, \ell_3\in\mathbbm L$. Here, $d(.,.)$
    denotes the graph distance and $\ell\in\mathbbm T$ is a
    \emph{leaf} if $\mathbbm T\setminus\{\ell\}$ only has a single
    connected component. Note that for such an ultra-metric tree there
    is a unique $r\in\mathbbm T$ (called the root) such that
    $d(r, \ell)$ does not depend on $\ell \in\mathbbm L$.  In
    addition, there is an order on $\mathbbm T$ such that $s\leq t$
    iff $s \in [r,t]$, where $[r,t]$ is the unique path from $r$ to
    $t$.  Then, we also say that $s$ is ancestor of $t$.  Using this
    order the most recent common ancestor $s\wedge t \in \mathbbm T$
    of $s$ and $t$ is the largest element in $\mathbbm T$ which is
    ancestor of both, $s$ and $t$.
  \item Usually, a Markov process $X = (X_t)_{t\geq 0}$ has the
    property that $(X_t)_{t\geq s}$ is independent of
    $(X_t)_{t\leq s}$ conditional on $X_s$. 
    Moreover, we call a Markov-process time-homogeneous if
    there is a family of transition kernels
    $(p_s(.,.))_{s\geq 0}$ such that
    \begin{align}
      \label{eq:markov}
      \mathbf P(X_t\in A|X_r=x) = p_{t-r}(x,A) \text{ for all }r\leq t.
    \end{align}
    If a Markov-process is piecewise constant, it is usually called a Markov jump process. 
    If it is time-homogeneous, it follows from \eqref{eq:markov} that the waiting time to 
    the next jump has an exponential distribution. Its parameter depends on the current
    state and is usually referred to as the {\it rate} of the exponential distribution.
    We call a tree-indexed process $X = (X_t)_{t\in \mathbbm T}$
    time-homogeneous Markov if~\eqref{eq:markov} holds for all
    $r,t\in\mathbbm T$ with $r\leq t$, where $t-r := d(r,t)$, and, 
    conditional on the value of the process at a node, the processes in the
    two descending subtrees are independent. For more
    work on Markov processes indexed by trees, see e.g.\
    \cite{benjamini}.
  \end{asparaenum}
\end{remark}

\begin{definition}[The ordered independent loss model along an ultrametric tree\label{def:oilossT}]
  Let $\mathbbm T$ be an ultra-metric tree with root $r\in\mathbbm T$.
  The $\mathbbm T$-indexed time-homogeneous Markov process
  $\mathcal S = (S_t)_{t\in\mathbbm T}$ with $S_r\sim\pi$ (recall from
  Lemma~\ref{l:seq}) and transition kernels given through gain and
  loss events as in Definition~\ref{def:oiloss}, is denoted the
  \emph{ordered independent loss model (along $\mathbbm T$)}.
\end{definition}

\begin{remark}
 Note that it is straightforward to formulate the presented results for non-ultrametric trees.
 However, since the formulas simplify significantly for ultrametric trees, we assume that $\mathbb T$ is ultrametric.
\end{remark}

In the sequel, we fix the ultra-metric tree $\mathbbm T$, its set of
leaves $\mathbbm L$ and the process $\mathcal S$ from
Definition~\ref{def:oilossT}. Recall that $S_t \in [0,1]^{\mathbbm N}$
and we will use the shorthand notation
$$ S_{t,i} := (S_t)_i.$$
Moreover, we will identify the vector $S_t$ with the set of its
entries, i.e. $S_t = \{S_{t, i}: i=1,2,\dots\}$. Note that all entries
of $S_t$ are different almost surely.

\begin{definition}[Equal spacers\label{def:eqsp}]
  For $\mathbbm K\subseteq \mathbbm L$ and $\ell\in\mathbbm K$, define
  recursively
  \begin{align*}
    E^{\ell, \mathbbm K}_{0} 
    & := E^{\ell, \mathbbm K, (\mathbbm L)}_{0} := 0,
    \\
    E^{\ell, \mathbbm K}_{n+1} 
    & := E^{\ell, \mathbbm K, (\mathbbm L)}_{n+1} := \min\Big\{i > E^{\ell, \mathbbm K}_{n}: S_{\ell,i} \in \bigcap_{k\in\mathbbm K} S_k 
      \setminus \Big(\bigcup_{k'\in\mathbbm L\setminus \mathbbm K} S_{k'}\Big) \Big\}, \qquad n=1,2,\dots
  \end{align*}
  Here, $E^{\ell, \mathbbm K}_{n}$ is the spacer position in $\ell$ of
  the $n$th spacer which is also contained in all
  $S_k, k\in\mathbbm K$, but in none of
  $S_{k'}, k'\in \mathbbm L\setminus\mathbbm K$.
\end{definition}

An illustration of the process from Definition \ref{def:oilossT} and the corresponding equal spacers are shown in Figure \ref{fig:oilossT}.

\begin{figure}
\centering
 \includegraphics{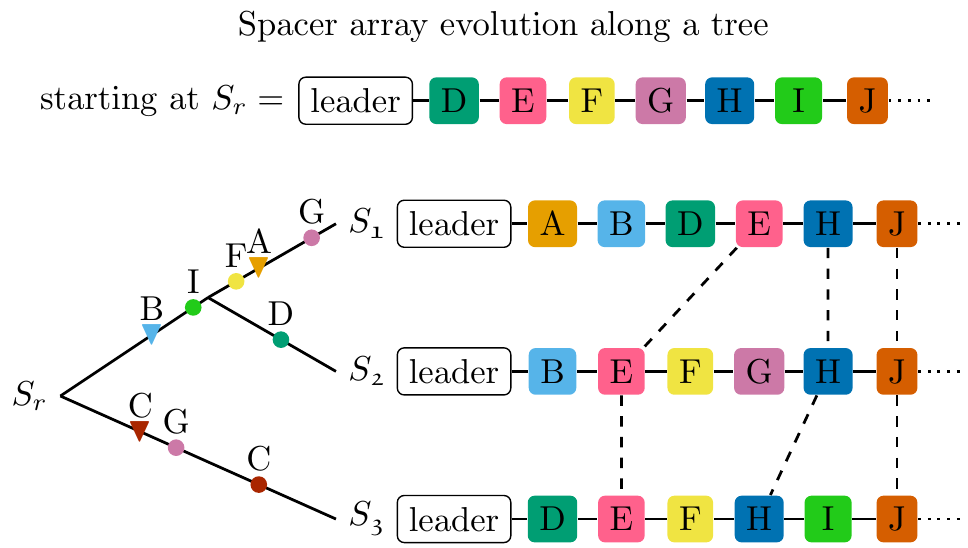}
 \caption{
 \label{fig:oilossT}
 Illustration of the ordered independent loss model along a tree from Definition \ref{def:oilossT}.
 New spacers (A-C) are inserted at the leader position at rate $\theta$ along the branches, marked by $\blacktriangledown$.
 Present spacers (A-J $\cdots$) are lost at rate $\rho$, marked by $\bullet$. Dashed lines connect the first three equal spacers shared between all arrays.
 Here $\mathbbm L = \{\mathfrak 1, \mathfrak 2, \mathfrak 3\}$ and from Definition \ref{def:eqsp} we get, 
 $E_{1}^{\mathfrak 2, \mathbbm L} = E_{1}^{\mathfrak 3, \mathbbm L} = 2$,
 $E_{1}^{\mathfrak 1, \mathbbm L} = E_{2}^{\mathfrak 3, \mathbbm L} = 4$,
 $E_{2}^{\mathfrak 1, \mathbbm L} = E_{2}^{\mathfrak 2, \mathbbm L} = 5$
 and $E_{3}^{\mathfrak 1, \mathbbm L} = E_{3}^{\mathfrak 2, \mathbbm L} = E_{3}^{\mathfrak 3, \mathbbm L} = 6$.
 }
\end{figure}

\section{Results}
\label{S:3}
\subsection[Trees with two leaves]{Trees with two leaves $|\mathbbm L|=2$}
In the special case that $\mathbbm L$ consists of only two points, we
denote these leaves by $\mathfrak 1$ and $\mathfrak 2$. In addition,
we set $d(\mathfrak 1, \mathfrak 2)=2T$ for some $T\geq 0$ and define
the following random variables:
\begin{align*}
  V_i & := E_{i}^{\mathfrak 1, \{\mathfrak 1, \mathfrak 2\}}, \qquad 
        W_i := E_{i}^{\mathfrak 2, \{\mathfrak 1, \mathfrak 2\}}, \qquad i=1,2,\dots,
\end{align*}
i.e.\ $V_i$ is the $i$th element of $S_{\mathfrak 1}$ which is also
contained in $S_{\mathfrak 2}$ and $W_i$ is the $i$th element of
$S_{\mathfrak 2}$ which is also contained in $S_{\mathfrak 1}$.


\begin{theorem}[Distribution of equal spacer sequence in two leaves\label{T:1}]
  Let $(A,B), (A_1, B_1), (A_2, B_2),\dots$ be iid pairs of random
  variables with joint distribution
  \begin{align}\label{eq:AB}
    \mathbb P(A=a, B=b) & = \binom{a+b}{a} \cdot \frac{e^{- \rho T}}{2-e^{-\rho T}} 
                          \cdot \bigg( \frac{1-e^{- \rho T}}{2-e^{-\rho T}}  \bigg)^{a+b}
  \end{align}
  In addition, let $C_{\mathfrak 1}, C_{\mathfrak 2}$ be iid Poisson
  distributed r.v.\ with parameter $\frac\theta\rho(1-e^{-\rho T})$.
  Then, 
  $$(V_1, W_1), (V_2-V_1, W_2-W_1), (V_3-V_2, W_3-W_2),\dots$$ are
  independent with
  \begin{align*}
    (V_1, W_1) \sim (C_{\mathfrak 1} + A_1, C_{\mathfrak 2} + B_1), \qquad
    (V_i-V_{i-1}, W_i-W_{i-1}) \sim (A_i, B_i), \qquad i=2,3,\dots
  \end{align*}
\end{theorem}

\begin{remark}[$A$ and $B$ are geometrically
  distributed\label{rem:geo1}]
  \begin{asparaenum}
    \item In the sequel, we will use the identity
      \begin{align}
        \label{eq:sumx}
        \sum_{b=0}^\infty \binom{a+b}{a} x^b = \frac{1}{(1-x)^{a+1}}, \qquad i=0,1,2,\dots
      \end{align}
      for $|x|<1$ on several occasions. It can easily be proven by
      induction.
    \item For the distribution of $(A, B)$ from~\eqref{eq:AB}, we note
      that both, $A$ and $B$ are geometrically distributed with
      parameter
      $e^{-\rho T}$. (We come back to this observation in Remark~\ref{rem:geo2}.)\\
      Indeed: For
      \begin{align}
        \label{def:x}
        x = \frac{1-e^{-\rho T}}{2-e^{-\rho T}} = 1 - \frac{1}{2-e^{-\rho T}},
      \end{align}
      we have that
      \begin{align*}
        \mathbbm P(A=a) & = \frac{e^{- \rho T}}{2-e^{-\rho T}}
                          \sum_{b=0}^\infty \binom{a+b}{a} x^{a+b}
                          = \frac{e^{-\rho T}}{2-e^{-\rho T}} \frac{1}{1-x}\Big(\frac{x}{1-x}\Big)^{a}
                          = e^{-\rho T}(1-e^{-\rho T})^a.
      \end{align*}
      So, we have shown that $A$ has the desired distribution. Symmetry
      then gives the same for $B$.
    \end{asparaenum}
\end{remark}

\noindent
Before we come to the proof of Theorem~\ref{T:1}, we state the
sampling formula which results from Theorem~\ref{T:1}.

\begin{corollary}[Sampling formula for equal spacer sequence in two
  leaves\label{cor:1}]
  The joint distribution of $(V_1, V_2,\dots,V_m, W_1, W_2, \dots, W_m)$ is given
  by
  \begin{align*}
    \mathbbm P( & V_1=v_1, \dots, V_m=v_m, W_1=w_1,\dots,W_m=w_m) 
    \\ & = 
         \bigg(\Big(\frac{\theta}{\rho}(1-e^{-\rho T})\Big)^{v_1+w_1} 
         e^{-2\frac{\theta}{\rho}(1-e^{-\rho T})} \, \frac{e^{-\rho T}}{2-e^{-\rho T}} 
    \\ & \qquad 
         \cdot\sum_{v=0}^{v_1} \sum_{w=0}^{w_1} \binom{v+w}{v} 
         \frac{1}{(v_1-v)!(w_1-w)!}\frac{1}{(\frac{\theta}{\rho}(2-e^{-\rho T}))^{v+w}}\bigg)
    \\ & \qquad \qquad \cdot 
         \bigg(\prod_{i=2}^m \binom{\Delta v_i + \Delta w_i}{\Delta v_i}\bigg)
         \cdot \bigg(\frac{e^{-\rho T}}{2-e^{-\rho T}} 
         \bigg)^{m-1}
         \cdot \bigg(\frac{1-e^{-\rho T}}{2 - e^{-\rho T}}
         \bigg)^{\sum_{i=2}^m \Delta v_i + \Delta w_i}
  \end{align*}
  for $m=1,2,\dots$, where
  $\Delta v_i := v_i - v_{i-1}, \Delta w_i := w_i - w_{i-1},
  i=2,3,\dots$
\end{corollary}

\begin{proof}
  By the independence from Theorem~\ref{T:1} and the distribution
  given in~\eqref{eq:AB}, the only thing which remains to be proven is
  that
  \begin{align*}
    \mathbbm P( V_1=v_1, W_1=w_1) 
    & = \Big(\frac{\theta}{\rho}(1-e^{-\rho T})\Big)^{v_1+w_1} 
      e^{-2\frac{\theta}{\rho}(1-e^{-\rho T})} \, \frac{e^{-\rho T}}{2-e^{-\rho T}} 
    \\ & \qquad \qquad \qquad 
         \cdot\sum_{a=0}^{v_1} \sum_{b=0}^{w_1} \binom{a+b}{a} 
         \frac{1}{(v_1-a)!(w_1-b)!}\frac{1}{(\frac{\theta}{\rho}(2-e^{-\rho T}))^{a+b}}.
  \end{align*}
  From Theorem~\ref{T:1}, we know that
  $(V_1, W_1) \sim (C_{\mathfrak 1} + A, C_{\mathfrak 2} + B)$, where
  the distribution of $(A,B)$ is as in \eqref{eq:AB} and
  $C_{\mathfrak 1}$ and $C_{\mathfrak 2}$ are independent and Poisson
  distributed with mean $z:=\frac\theta\rho(1-e^{-\rho T})$.
  Hence, with $x$ from \eqref{def:x},
  \begin{align*}
    \mathbbm P(& V_1=v_1,W_1=w_1) 
                 = \sum_{a=0}^{v_1} \sum_{b=0}^{w_1} \mathbbm P(A = a, 
                 C_{\mathfrak 1}=v_1-a, B=b, C_{\mathfrak 2} = w_1-b)
    \\ &= \sum_{a=0}^{v_1} \sum_{b=0}^{w_1} \mathbbm P(A = a, B=b) \cdot 
         \mathbbm P(C_{\mathfrak 1}=v_1-a) \cdot \mathbbm P(C_{\mathfrak 2}=w_1-b)
    \\ &=\sum_{v=0}^{v_1} \sum_{w=0}^{w_1} \binom{a+b}{a} \frac{e^{-\rho T}}{2-e^{-\rho T}}
         x^{a+b}
         \frac{z^{v_1+w_1-(a+b)}}{(v_1-a)!(w_1-b)!} e^{-2z}\\
               &= z^{v_1  +w_1} \frac{e^{-2z}e^{-\rho T}}{2 - e^{-\rho T}} 
                 \sum_{a=0}^{v_1} \sum_{b=0}^{w_1} \binom{a+b}{a} \frac{1}{(v_1-a)!(w_1-b)!} 
                 \Big(\frac{x}{z}\Big)^{a+b}.
  \end{align*}
  Plugging in $x$ and $z$ gives the result.
\end{proof}

\begin{proof}[Proof of Theorem~\ref{T:1}]
  For the leaves $\mathfrak 1$ and $\mathfrak 2$, recall that
  $\mathfrak 1\wedge\mathfrak 2$ denotes their MRCA (most recent
  common ancestor) and $d(\mathfrak 1, \mathfrak 1\wedge\mathfrak 2) = T$. Let $C_{\mathfrak 1} (C_{\mathfrak 2})$ be the
  number of gain-events between $\mathfrak 1\wedge\mathfrak 2$ and
  $\mathfrak 1$ ($\mathfrak 2$), which don't get lost until
  $\mathfrak 1$ (until $\mathfrak 2$). Then, by construction,
  $C_{\mathfrak 1}$ and $C_{\mathfrak 2}$ are independent (since they
  depend on independent gain events) and Poisson distributed with mean
  $$ \int_{0}^{T} \theta e^{-\rho t}  
  dt = \frac\theta\rho(1-e^{-\rho T}).$$
  For $(S_{\mathfrak 1, C_{\mathfrak 1} + i})_{i=1,2,\dots}$ and
  $(S_{\mathfrak 2, C_{\mathfrak 2} + i})_{i=1,2,\dots}$, i.e.\ the
  spacers after the just mentioned gain-events, we note that, by
  construction,
  $$ \{S_{\mathfrak 1, C_{\mathfrak 1} + i}: i=1,2,\dots\} \cup 
  \{S_{\mathfrak 2, C_{\mathfrak 2} + i}: i=1,2,\dots\} \subseteq
  S_{\mathfrak 1 \wedge\mathfrak 2}.$$
  Moreover, we have that (by independence of loss-events at all
  positions) the events
  $$ (\{s \in S_{\mathfrak 1}\}, \{s\in S_{\mathfrak 2}\})_{s\in S_{\mathfrak 1 \wedge\mathfrak 2}}$$
  are independent with
  $$ \mathbbm P(s\in S_{\mathfrak 1}) =  \mathbbm P(s\in S_{\mathfrak 2}) = e^{-\rho T}$$
  for all $s\in S_{\mathfrak 1 \wedge\mathfrak 2}$. Equivalently, we
  find that independently for all
  $s\in S_{\mathfrak 1 \wedge \mathfrak 2}$,
  \begin{equation}\label{eq:p14}
    \begin{aligned}
      p_1 &:= \mathbbm P(s\in S_{\mathfrak 1}\setminus S_{\mathfrak
        2}) = e^{-\rho T}(1-e^{-\rho T}),
      \\
      p_2 & := \mathbbm P(s\in S_{\mathfrak 2}\setminus S_{\mathfrak
        1}) = e^{-\rho T}(1-e^{-\rho T}),
      \\
      p_3 & := \mathbbm P(s\notin S_{\mathfrak 2}\cup S_{\mathfrak 1})
      = (1-e^{-\rho T})^2,
      \\
      p_4 & := \mathbbm P(s\in S_{\mathfrak 2}\cap S_{\mathfrak 1}) =
      e^{-2\rho T}.
    \end{aligned}
  \end{equation}
  Moving along $S_{\mathfrak 1\wedge \mathfrak 2}$, toss a four-sided
  die, numbered $1,\dots,4$. If it comes up~1, keep the spacer in
  $\mathfrak 1$ but not in $\mathfrak 2$; if it comes up~2, keep it in
  $\mathfrak 2$ but not in $\mathfrak 1$; if it comes up~3, neither
  keep it in $\mathfrak 1$ nor in $\mathfrak 2$; if it comes up~4,
  keep it in both, $\mathfrak 1$ and $\mathfrak 2$. Then, we need to
  compute the joint distribution of the number of die rolls $A$
  with~1 and $B$ with~2 before the first~4 (which indicates an equal
  spacer). Letting
  $K$ be the first die roll with~4, we have that
  \begin{align*}
    \mathbbm P&(A = a, B=b) = \sum_{k = a+b+1}^\infty \mathbbm P(A=a, B=b, K=k)\\
              &= \sum_{k = a+b+1}^\infty \mathbbm P(K=k)\mathbbm P(A=a, B=b | K=k) \\
              &=\sum_{k = a+b+1}^\infty p_4 (1-p_4)^{k-1} \binom{k-1}{a+b}\binom{a+b}{a} 
                \left(\frac{p_1}{1-p_4}\right)^{a}\left(\frac{p_2}{1-p_4}\right)^{b}\left(\frac{p_3}{1-p_4}\right)^{k-1-(a+b)}\\
              &=\binom{a+b}{a} p_1^{a} p_2^{b} p_4 \sum_{k=a+b+1}^\infty \binom{k-1}{a+b} p_3^{k-1-(a+b)}\\
              &=\binom{a+b}{a} p_1^{a} p_2^{b} p_4 \sum_{k=0}^\infty \binom{a+b+k}{a+b} p_3^{k}\\
              &=\binom{a+b}{a} \frac{p_4}{1-p_3} \Big(\frac{p_1}{1-p_3}\Big)^{a+b} \\
              & = \binom{a+b}{a} \frac{e^{-2\rho T}}{1-(1-e^{-\rho T})^2} 
                \Big(\frac{e^{-\rho T}(1-e^{-\rho T})}{1-(1-e^{-\rho T})^2}\Big)^{a+b}\\
              & = \binom{a+b}{a} \frac{e^{-\rho T}}{2 - e^{-\rho T}} 
                \Big(\frac{1-e^{-\rho T}}{2 - e^{-\rho T}}\Big)^{a+b}
  \end{align*}
  since $p_1=p_2$, where we have used \eqref{eq:sumx} in the second to
  last equality. Before the first equal spacer, we have for
  $(V_1, W_1)$ the sum of new and old spacers, while for
  $(V_i, W_i), i\geq 1$,
  we only have old spacers. Since loss events of spacers in
  $\mathfrak 1\wedge \mathfrak 2$ are independent,
  the result follows.
\end{proof}

\begin{remark}[$\Delta V_i$ and $\Delta W_i$ are geometrically
  distributed\label{rem:geo2}]
  We have seen in the proof of Theorem~\ref{T:1} that the distribution
  of $(\Delta V_i, \Delta W_i)$ for $i=2,3,\dots$ can be obtained by
  rolling a die with probabilities $p_1,\dots,p_4$ given by
  \eqref{eq:p14} and counting the number of occurrences of~1 and~2 before
  the first~4. For the marginal distribution of $\Delta V_i$, we are
  asking for the number of occurrences of~1 before the
  first~4. Clearly, this is geometrically distributed with success
  probability $p_4/(p_1+p_4) = e^{-\rho T}$. See also the
  formal calculation in Remark~\ref{rem:geo1}.
\end{remark}

\subsection[Trees with three leaves]{Trees with three leaves $|\mathbbm L|=3$}
In the special case that $\mathbbm L$ consists of three points,
denoted $\mathfrak 1, \mathfrak 2$ and $\mathfrak 3$, we define the
following random variables:
\begin{align*}
  X_i & := E_{i}^{\mathfrak 1, \{\mathfrak 1, \mathfrak 2, \mathfrak 3\}}, \qquad 
        Y_i := E_{i}^{\mathfrak 2, \{\mathfrak 1, \mathfrak 2, \mathfrak 3\}}, \qquad 
        Z_i := E_{i}^{\mathfrak 3, \{\mathfrak 1, \mathfrak 2, \mathfrak 3\}}, \qquad i=0,1,2,\dots,
\end{align*}
i.e.\ $X_i$ is the $i$th element of $S_{\mathfrak 1}$ which is also
contained in both, $S_{\mathfrak 2}$ and $S_{\mathfrak 3}$, $Y_i$ is
the $i$th element of $S_{\mathfrak 2}$ which is also contained in
$S_{\mathfrak 1}$ and $S_{\mathfrak 3}$, and $Z_i$ is the $i$th
element of $S_{\mathfrak 3}$ which is also contained in
$S_{\mathfrak 1}$ and $S_{\mathfrak 2}$. In between $X_i$ and
$X_{i+1}$, for example, we find spacers of three classes: those, which are
only in $S_{\mathfrak 1}$, i.e.\ in
$S_{\mathfrak 1} \setminus (S_{\mathfrak 2}\cup S_{\mathfrak 3})$,
those which are shared with $S_{\mathfrak 2}$, i.e.\ in
$S_{\mathfrak 1} \cap S_{\mathfrak 2} \setminus S_{\mathfrak 3}$, and
those shared with $S_{\mathfrak 3}$, i.e.\ in
$S_{\mathfrak 1} \cap S_{\mathfrak 3} \setminus S_{\mathfrak 2}$.
Recalling the notation from Definition~\ref{def:eqsp}, we write for
$\mathbbm K\subseteq\mathbbm L = \{\mathfrak 1, \mathfrak 2, \mathfrak
3\}$
\begin{align*}
  F_i^{\mathbbm K}
  & := \# \{E_i^{k, \mathbbm L},E_i^{k, \mathbbm L}\!\!+\!1,\dots,E_{i+1}^{k, \mathbbm L}\} \cap 
    \{E_n^{k, \mathbbm K}: n=1,2,\dots\}, \qquad i=0,1,2,...
\end{align*}
where the right hand side does not depend on which $k\in\mathbbm K$ we
choose, because the order is not altered by losses.
In words, $F_i^{\mathbbm K}$ is the number of spacers within
$S_k$, which are between the $i$th and $(i+1)$st spacer shared among
$\mathfrak 1, \mathfrak 2, \mathfrak 3$ and that appear in all elements of
$\mathbbm K$, but in no element of $\mathbbm L\setminus \mathbbm
K$. As an example, we can rewrite
\begin{align*}
  F_i^{\{\mathfrak 1\}} & =\# \{S_{\mathfrak 1, X_i},S_{\mathfrak 1, X_i+1},\dots,S_{\mathfrak 1, X_{i+1}}\} 
                          \cap (S_{\mathfrak 1} \setminus ( S_{\mathfrak 2} \cup S_{\mathfrak 3})),
  \\ 
  F_i^{\{\mathfrak 1,  \mathfrak 2\}} & = \# \{S_{\mathfrak 1, X_i},S_{\mathfrak 1, X_i+1},\dots,S_{\mathfrak 1, X_{i+1}}\} 
                                        \cap ((S_{\mathfrak 1} \cap S_{\mathfrak 2})\setminus S_{\mathfrak 3}),
  \\ & = \# \{S_{\mathfrak 2, Y_i},S_{\mathfrak 2, Y_i+1},\dots,S_{\mathfrak 2, Y_{i+1}}\} 
       \cap ((S_{\mathfrak 1} \cap S_{\mathfrak 2})\setminus S_{\mathfrak 3}).
\end{align*}
Note that
$X_{i+1} - X_i = F_i^{\mathfrak 1 \setminus (\mathfrak 2, \mathfrak
  3)} + F_i^{(\mathfrak 1 \cup \mathfrak 2)\setminus \mathfrak 3} +
F_i^{(\mathfrak 1 \cup \mathfrak 3)\setminus \mathfrak 2} + 1$
by construction, $i=0,1,2,\dots$ 

Since there is exactly one tree topology of a tree with three leaves,
we will pick an ultra-metric tree, where $\mathfrak 1$ and
$\mathfrak 2$ are closer relatives than $\mathfrak 1$ and
$\mathfrak 3$ (and therefore also closer than $\mathfrak 2$ and
$\mathfrak 3$). So, $T, T'$ denote the distances from $\mathfrak 1$ to
$\mathfrak 3$ and from $\mathfrak 1$ to $\mathfrak 2$, respectively,
with $T\geq T'$. See Figure \ref{fig:1} for an illustration of the
tree topology.

\begin{figure}
\centering
 \begin{tikzpicture}
  \draw (1,0) -- (2,1.5) -- (3,3);
  \draw (3,0) -- (2,1.5);
  \draw (5,0) -- (3,3);
  \node[below] at (1,0){$\mathfrak 1$};
  \node[below] at (3,0){$\mathfrak 2$};
  \node[below] at (5,0){$\mathfrak 3$};
  \draw (6.1,0) -- (5.9,0) -- (6,0)  -- (6,3) -- +(0.1,0) -- +(-0.1,0);
  \node[left] at (0,0.75){$T'$};
  \draw (0.1,0) -- (-0.1,0) -- (0,0)  -- (0,1.5) -- +(0.1,0) -- +(-0.1,0);
  \node[left] at (6,1.5){$T$};
 \end{tikzpicture}
 \caption{\label{fig:1}The tree topology for three leaves is chosen
   such that $\mathfrak 1, \mathfrak 2$ are closer related than
   $\mathfrak 1, \mathfrak 3$ and $\mathfrak 2, \mathfrak 3$. }
\end{figure}

\begin{theorem}[Distribution of equal spacer sequence in three leaves\label{T:2}]
  For $\mathbbm L = \{\mathfrak 1, \mathfrak 2, \mathfrak 3\}$, let
  $(A^{\mathbbm K})_{\emptyset\subsetneq \mathbbm K \subsetneq
    \mathbbm L}, (A_1^{\mathbbm K})_{\emptyset\subsetneq \mathbbm K
    \subsetneq \mathbbm L}, (A_2^{\mathbbm K})_{\emptyset\subsetneq
    \mathbbm K \subsetneq \mathbbm L},\dots$
  be iid tuples of random variables with joint distribution
  \begin{equation}
    \label{eq:ABCDEF}
    \begin{aligned}
      &P(A^{\{\mathfrak1\}} = a, A^{\{\mathfrak2\}} = b,
      A^{\{\mathfrak3\}} = c,
      A^{\{\mathfrak1,\mathfrak2\}} = d, A^{\{\mathfrak1,\mathfrak3\}} = e, A^{\{\mathfrak2,\mathfrak3\}} = f)\\
      &= \binom{a+b+c+d+e+f}{a,b,c,d,e,f} \Big(\frac{(1-e^{-\rho
          T})(1-e^{-\rho T'})}{r}\Big)^{a+b}
      \Big(\frac{1-2e^{-\rho T} + e^{-\rho 
          (T+T')})}{r}\Big)^{c} \\ & \qquad\qquad \qquad\qquad
      \qquad\qquad \cdot \Big(\frac{e^{-\rho T'}(1-
        e^{-\rho T})}{r}\Big)^{d} \Big(\frac{e^{-\rho 
          T}(1- e^{-\rho T'})}{r}\Big)^{e+f}
      \frac{e^{-\rho T}e^{-\rho T'}}{r}
    \end{aligned}
  \end{equation}
  with 
  $$ r := 3 - e^{-\rho T'} - e^{-\rho T}(2-e^{-\rho T'}).$$ 
  In addition, let
    \begin{align*}
      &C^{\{\mathfrak 1\}} \stackrel d = \text{Poi}\left(\frac{\theta}{\rho}(1-e^{-\rho T'})
        (1+e^{-\rho T'}-e^{-\rho T})\right),
      &&  C^{\{\mathfrak 1, \mathfrak 2\}} \stackrel d = \text{Poi}\left(\frac{\theta}{\rho}
         (1-e^{- \rho (T-T')}) e^{-2\rho T'}\right),
      \\  &C^{\{\mathfrak 2\}} \stackrel d = \text{Poi}\left(\frac{\theta}{\rho}(1-e^{-\rho T'})
         (1+e^{-\rho T'}-e^{-\rho T})\right),      
      &&  C^{\{\mathfrak 1, \mathfrak 3\}} = 0,
      \\ & C^{\{\mathfrak 3\}} \stackrel d = \text{Poi}\left(\frac{\theta}{\rho}
            (1-e^{-\rho T})\right),      
      &&  C^{\{\mathfrak 2, \mathfrak 3\}} = 0
   \end{align*}
  be independent. Then,
  $$ (F_0^{\mathbbm K})_{\emptyset \subsetneq \mathbbm K\subsetneq \mathbbm L}, 
  (F_1^{\mathbbm K})_{\emptyset \subsetneq \mathbbm K\subsetneq
    \mathbbm L}, (F_2^{\mathbbm K})_{\emptyset \subsetneq \mathbbm
    K\subsetneq \mathbbm L},\dots$$ are independent with
  \begin{align*}
    (F_0^{\mathbbm K})_{\emptyset \subsetneq \mathbbm K\subsetneq \mathbbm L} 
    & \stackrel d = (C^{\mathbbm K} + A^{\mathbbm K})_{\emptyset \subsetneq \mathbbm K\subsetneq \mathbbm L},
    \\ 
    (F_i^{\mathbbm K})_{\emptyset \subsetneq \mathbbm K\subsetneq \mathbbm L} 
    & \stackrel d = (A_i^{\mathbbm K})_{\emptyset \subsetneq \mathbbm K\subsetneq \mathbbm L}, \qquad i=1,2,...
  \end{align*}
\end{theorem}

\begin{remark}[Old and new spacers]
  \begin{asparaenum}
  \item In the proof of the Theorem, we will divide the spacers into
    {\em old} spacers, which were already present in
    $\mathfrak 1\wedge \mathfrak 2 \wedge \mathfrak 3$, the most
    recent common ancestor of $\mathfrak 1, \mathfrak 2$ and
    $\mathfrak 3$, and {\em new} spacers, which were gained after
    $\mathfrak 1\wedge \mathfrak 2 \wedge \mathfrak 3$.  As in the
    proof of Theorem~\ref{T:1}, we will argue that every old spacer
    has a chance to be kept until $\mathfrak 1, \mathfrak 2$ and
    $\mathfrak 3$. Here, we have to keep in mind that losses in
    $\mathfrak 1$ and $\mathfrak 2$ are not independent due to the
    chosen tree topology.
  \item Although we formulate Theorem~\ref{T:2} analogously to
    Theorem~\ref{T:1} for two leaves, there are some conceptual
    differences. In particular, in the case of two leaves,
    $\mathfrak 1$ and $\mathfrak 2$, it is clear that only after the
    first equal spacers we can be sure that spacers are {\em old} in
    the sense that they are also contained in
    $\mathfrak 1\wedge\mathfrak 2$. In the case of three leaves,
    $\mathfrak 1, \mathfrak 2$ and $\mathfrak 3$, and the
    tree-topology from above, we know that the number of {\em new}
    spacers shared between $\mathfrak 1, \mathfrak 3$, and between
    $\mathfrak 2, \mathfrak 3$, is zero (since
    $C^{\{\mathfrak 1, \mathfrak 3\}} =C^{\{\mathfrak 2, \mathfrak
      3\}} = 0$).
    Hence, if the first spacer in
    $(S_{\mathfrak 1}\cap S_{\mathfrak 3})\cup (S_{\mathfrak 2}\cap
    S_{\mathfrak 3})$
    arises, we know that subsequent spacers must be {\em old}. In
    particular,
    $(S_{\mathfrak 1,1},\dots,S_{\mathfrak 1, X_1}), (S_{\mathfrak
      2,1},\dots,S_{\mathfrak 2, Y_1}), (S_{\mathfrak
      3,1},\dots,S_{\mathfrak 3, Z_1})$
    contain some information about which spacers are new and old,
    which cannot be gathered from
    $F^{\{\mathfrak 1\}}_1, F^{\{\mathfrak 2\}}_1, F^{\{\mathfrak
      3\}}_1$.
  \end{asparaenum}
\end{remark}

\begin{proof}[Proof of Theorem~\ref{T:2}]

  For the leaves $\mathfrak 1$ and $\mathfrak 2$, we denote their MRCA
  by $\mathfrak 1\wedge\mathfrak 2$ and $d(\mathfrak 1, \mathfrak 1 \wedge \mathfrak 2) = T'$.

  Let
  $C_{\mathfrak 1} (C_{\mathfrak 2})$ be the number of gain-events
  between $\mathfrak 1\wedge\mathfrak 2$ and $\mathfrak 1$
  ($\mathfrak 2$), which don't get lost until $\mathfrak 1$ (until
  $\mathfrak 2$). Then, by construction, $C_{\mathfrak 1}$ and
  $C_{\mathfrak 2}$ are independent (since they depend on independent
  gain events) and Poisson distributed with mean
  $$ \int_{0}^{T'} \theta  e^{-\rho t}  
  dt = \frac\theta\rho(1-e^{-\rho T'}).$$
  For $(S_{\mathfrak 1, C_{\mathfrak 1} + i})_{i=1,2,\dots}$ and
  $(S_{\mathfrak 2, C_{\mathfrak 2} + i})_{i=1,2,\dots}$, i.e.\ the
  spacers after the just mentioned gain-events, we note that, by
  construction,
  $$ \{S_{\mathfrak 1, C_{\mathfrak 1} + i}: i=1,2,\dots\} \cup 
  \{S_{\mathfrak 2, C_{\mathfrak 2} + i}: i=1,2,\dots\} \subseteq
  S_{\mathfrak 1 \wedge\mathfrak 2}.$$

  We take the same route as in the proof of Theorem \ref{T:1}. First,
  consider the new spacers, i.e.\ spacers gained after
  $\mathfrak 1\wedge \mathfrak 2\wedge \mathfrak 3$. Note that there
  are no new spacers in
  $F_0^{\{\mathfrak 1, \mathfrak 3\}}, F_0^{\{\mathfrak 2, \mathfrak
    3\}}$
  (hence we set
  $C^{\{\mathfrak 1, \mathfrak 3\}} = C^{\{\mathfrak 2, \mathfrak
    3\}}=0$).
  Since gain-events follow a Poisson process, the number of new
  spacers in
  $(F_0^{\mathbbm K})_{\emptyset \subsetneq \mathbbm K\subsetneq L}$
  are independent and have a Poisson distribution which
  is uniquely determined by their mean. As an example,
  consider $F_0^{\{\mathfrak 1\}}$, here, we have to take into account
  spacers gained between
  $\mathfrak 1\wedge \mathfrak 2\wedge \mathfrak 3$ and
  $\mathfrak 1\wedge \mathfrak 2$, which are kept in $\mathfrak 1$ but
  lost in $\mathfrak 2$, and spacers gained between
  $\mathfrak 1\wedge \mathfrak 2$ and $\mathfrak 1$, which are not
  lost until $\mathfrak 1$. We compute 
  $$ \theta \int_0^{T-T'} e^{-\rho s}ds \cdot e^{-\rho T'}(1-e^{-\rho T'}) +  \theta \int_0^{T'}
  e^{-\rho s}ds = \frac\theta\rho (1-e^{-\rho T'})\big(
  (1 - e^{-\rho (T-T')})e^{-\rho T'} + 1\big)$$
  for the mean of $F_0^{\{\mathfrak 1\}}$, which corresponds to
  $C^{\{\mathfrak 1\}}$ above. The rates of all other
  classes of new spacers are obtained accordingly.

  Moving along the spacers in
  $S_{\mathfrak 1\wedge \mathfrak 2\wedge \mathfrak 3}$, we now
  distinguish the new spacers into eight cases (spacers kept only in $\mathfrak 1$,
  only in $\mathfrak 2$, only in $\mathfrak 3$, only in $\mathfrak 1$
  and $\mathfrak 2$, only in $\mathfrak 1$ and $\mathfrak 3$, only in
  $\mathfrak 2$ and $\mathfrak 3$, kept in none and kept in all). We
  count the number of die rolls with a $1,2,\dots,6$ until the first
  $8$ appears. The probabilities for the die are defined analogous to
  \eqref{eq:p14} as
  \begin{align*}
    q_1 &:= P(s\in S_\mathfrak1 \setminus (S_\mathfrak2 \cup S_\mathfrak3)) =  \ehalb(1-\etwohalb)(1-\ehalb)\\
    q_2 &:= P(s\in S_\mathfrak2 \setminus (S_\mathfrak1 \cup S_\mathfrak3)) = \ehalb(1-\etwohalb)(1-\ehalb)\\
    q_3 &:= P(s\in S_\mathfrak3 \setminus (S_\mathfrak1 \cup S_\mathfrak2)) = \ehalb(1-2\ehalb+\eonetwohalb)\\
    q_4 &:= P(s\in (S_\mathfrak1 \cap S_\mathfrak2) \setminus S_\mathfrak3) = \ehalb\etwohalb(1-\ehalb)\\
    q_5 &:= P(s\in (S_\mathfrak1 \cap S_\mathfrak3) \setminus S_\mathfrak2) = \ehalb\ehalb(1-\etwohalb)\\
    q_6 &:= P(s\in (S_\mathfrak2 \cap S_\mathfrak3) \setminus S_\mathfrak1) = \ehalb\ehalb(1-\etwohalb)\\
    q_7 &:= P(s\notin S_\mathfrak1 \cup S_\mathfrak2 \cup S_\mathfrak3) = (1-\ehalb)(1-2\ehalb+\eonetwohalb)\\
    q_8 &:= P(s\in S_\mathfrak1 \cap S_\mathfrak2 \cap S_\mathfrak3) = \ehalb\ehalb\etwohalb.
  \end{align*}
  Then we can write \eqref{eq:ABCDEF} as (note that
  $ r = e^{\rho T}(1-q_7)$)
  \begin{align*}
    &=\binom{a+b+c+d+e+f}{a,b,c,d,e,f}
      \Big(\frac{q_1}{1-q_7}\Big)^a
      \Big(\frac{q_2}{1-q_7}\Big)^b
      \Big(\frac{q_3}{1-q_7}\Big)^c
      \Big(\frac{q_4}{1-q_7}\Big)^d
      \Big(\frac{q_5}{1-q_7}\Big)^e
      \Big(\frac{q_6}{1-q_7}\Big)^f
      \frac{q_8}{1-q_7}
    \\&=\binom{a+b+c+d+e+f}{a,b,c,d,e,f}
        \Big(\frac{q_1}{1-q_7}\Big)^{a+b}
        \Big(\frac{q_3}{1-q_7}\Big)^c
        \Big(\frac{q_4}{1-q_7}\Big)^d
        \Big(\frac{q_5}{1-q_7}\Big)^{e+f}
        \frac{q_8}{1-q_7}
    \\&=\binom{a+b+c+d+e+f}{a,b,c,d,e,f}
        \Big(\frac{(1-e^{-\rho T})(1-e^{-\rho T'})}{r}\Big)^{a+b}
        \Big(\frac{1-2e^{-\rho T} + e^{-\rho (T+T')})}{r}\Big)^{c}
    \\ & \qquad\qquad \qquad\qquad \qquad\qquad
         \cdot \Big(\frac{e^{-\rho T'}(1- e^{-\rho T})}{r}\Big)^{d}
         \Big(\frac{e^{-\rho T}(1- e^{-\rho T'})}{r}\Big)^{e+f}
         \frac{e^{-\rho (T+T')}}{r}.
  \end{align*}
  Again, before the first equal spacer, we split
  $(F_0^{\mathbbm K})_{\emptyset \subsetneq \mathbbm K\subsetneq
    \mathbbm L}$
  into the sum of new and old spacers, while for
  $(F_i^{\mathbbm K})_{\emptyset \subsetneq \mathbbm K\subsetneq
    \mathbbm L}, i\geq 1$,
  we only have old spacers. Since loss events of spacers in
  $\mathfrak 1\wedge \mathfrak 2\wedge \mathfrak 3$ are independent,
  the result follows.
\end{proof}

As in Corollary~\ref{cor:1}, it is straight-forward to translate the
last Theorem into a sampling formula, i.e.\ a formula for the
distribution of the family of random variables
$$ (F_0^{\mathbbm K})_{\emptyset \subsetneq \mathbbm K\subsetneq \mathbbm L}, 
(F_1^{\mathbbm K})_{\emptyset \subsetneq \mathbbm K\subsetneq \mathbbm
  L}, (F_2^{\mathbbm K})_{\emptyset \subsetneq \mathbbm K\subsetneq
  \mathbbm L},\dots$$

\subsection[Trees with $n$ leaves]{Trees with $n$ leaves $|\mathbbm L|=n$}
For the general case with $n$ leaves the distribution of equal spacer
sequences is based on a recursion along the tree $\mathbb T$, which we
introduce in Definition \ref{def:survival}. Throughout this section, we fix the
tree $\mathbb T$.

\begin{definition}[Survival function]
  \label{def:survival}
  Let $\mathbb T$ be a binary tree with root $r$, a set of leaves
  $\mathbb L$ and internal vertices $\mathbb V$, including $r$. Recall that there is
  a semi-order on $\mathbb T$ such that $r$ is the smallest element,
  and $\mathbb L$ is the set of maximal elements. For an internal
  branch point $t \in \mathbb V$, we denote by $t_+$ and $t^+$ two
  points infinitesimally close to $t$, pointing in the two directions
  in $\mathbb T$ leading to bigger elements.  The survival function
  $p: \mathbb T \rightarrow [0,1]$ is defined by 
  \begin{align*}
    \frac{ \partial p(t) } {\partial t}  &= \rho\,p(t) &&\text{ for } t \in \mathbb T \setminus (\mathbb L \cup \mathbb V) \\
    p(t) &= 1 - ( (1-p(t_+))  (1-p(t^+))  )  &&\text{ for } t \in \mathbb V\phantom{\setminus (\mathbb L \cup \mathbb V)}
  \end{align*}
  using the initial condition $p(t) = 1 \text{ for } t \in \mathbb L.$
  In addition, we set $q(t) := 1-p(t), t\in\mathbb T$.
\end{definition}

Roughly speaking, $p(t)$ is the probability that a spacer present at
$t \in \mathbb T$ is present in at least one of the leaves of the subtree
starting in $t$. Analogously, $q(t)$ is the
probability that a spacer present in $t$ is lost in the whole subtree
starting in $t$.  This is formalized in the next Proposition.

\begin{proposition}[probability to loose a spacer along all paths from
  $t$ to the leaves]
 \label{prop:looseall}
 Let $\mathcal S$ be defined as in Definition \ref{def:oilossT}.  Then, for a
 spacer $s \in S_t$ the probability to be absent in all
 $l\in\mathbb L$, i.e.\ to get lost along each path from $t$ to
 $\mathbb L$, is
 \begin{equation}
   \mathbf P\Big[ s \notin \bigcup\limits_{l \in \mathbb L} S_l | s \in S_t  \Big] = q(t).
 \end{equation}
\end{proposition}

\begin{proof}
  The result follows directly from the definition of $p$.
  Note, however, that if $p(t_+)$ and $p(t^+)$
  are the probabilities to lose a spacer at least once in the subtree
  of $t_+$ and $t^+$, the probability to not lose it at all on the
  subtree of $t$ is $(1-p(t_+))(1-p(t^+))$.
  This explains the form of $p(t)$ for $t\in\mathbb V$.
  See also Proposition 5.2 in \cite{BaumdickerHessPfaffelhuber2010}.
\end{proof}

For $n > 3$ the tree topology may differ between trees of the same
size.  It is hence more involved to compute the distribution of spacer
sequences for a given tree $\mathbbm T$.  To formulate the result for
the general distribution of equal spacer sequences in $n$ leaves we
need further notation.

\begin{definition}
  For a tree $\mathbb T$ with root $r$,
  $n-2$ further
  internal vertices $v \in \mathbb V$, and
  a set of $n$ leaves,
  $l \in \mathbb L = \{\mathfrak 1,...,\mathfrak n\}$, 
  we define for any
  $\mathbb K \subseteq \mathbb L$ (we use the semi-order $\leq$ also
  set-valued, meaning that $t\leq \mathbb K$ means $t\leq l$ for all
  $l\in\mathbb K$)
  \begin{align*}
    v_{\mathbbm K} 
    & := \sup\{t\in\mathbb T: t\leq \mathbb K\}
      && \text{is the MRCA of all leaves in $\mathbbm K$,}\\
      \mathbb T_{\mathbb K} & := \{t\in \mathbb T: v_{\mathbb K} \leq t \leq \mathbb K\}
                              && \text{is the subtree spanning $\mathbb K$,}\\
                              \mathbb T_{\mathbb K}^r & := \{t\in \mathbb T: r\leq t \leq \mathbb K\}
                                                        && \text{is the same subtree, but including $[r,v_{\mathbb K})$.}
                                                        \intertext{For $v\in\mathbb V$, set}
                                                        \mathbb T^v & := \{t\in \mathbb T: t\geq v\},
                                                                      &&\text{the subtree above $v$,}\\
                                                                      \mathbb T^v_\mathbb K & := \mathbb T^v \cap \mathbb T_{\mathbb K}^r,
                                                                                              &&\text{and}\\
                                                                                              v_\downarrow & := \max\{w\in \mathbb V: w< v\}
												&&\text{the internal vertex prior to $v$ for $v \in \mathbb V$, $v\neq r.$}
  \end{align*}
  Moreover, let
  $\mathbb S\subseteq \mathbb T$ be a subtree. We define its length by
  $\lambda (\mathbb S)$. Note that $\mathbb T\setminus \mathbb S$ is a
  forest, which again consists of trees. So, we identify
  $\mathbb T\setminus \mathbb S$ with this set of trees. For such a
  tree $\mathbb S' \in \mathbb T\setminus \mathbb S$, we set
  $$q^\ast(\mathbb S') := q(\inf\{t\in \mathbb S'\}),$$
  which -- from Proposition~\ref{prop:looseall} -- is the probability
  that a spacer present at the root of $\mathbb S'$ is absent in all
  leaves of $\mathbb S'$.
\end{definition}

\begin{theorem}[The distribution of equal spacer sequences in $n$ leaves\label{T:n}]
  For $\mathbbm L = \{\mathfrak 1, \mathfrak 2,\dots, \mathfrak n\}$, let
  $\mathcal M := \mathcal P(\mathbb L) \setminus \{\mathbb L\}$ be the
  set of true subsets of $\mathbb L$. Let
  $(A^{\mathbbm K})_{\emptyset\subsetneq \mathbbm K \subsetneq
    \mathbbm L}, (A_1^{\mathbbm K})_{\emptyset\subsetneq \mathbbm K
    \subsetneq \mathbbm L}, (A_2^{\mathbbm K})_{\emptyset\subsetneq
    \mathbbm K \subsetneq \mathbbm L},\dots$
  be iid tuples of random variables with joint distribution
  \begin{equation}
    \label{eq:general}
    P( (A^{\mathbbm K} = a^{\mathbbm K})_{\mathbbm K \in \mathcal M} )
    = p(r)^{-(1+\sum\limits_{\mathbbm K \in \mathcal M} a^{\mathbbm K})}
    e^{-\rho \lambda(\mathbbm L)}
    \binom{\sum\limits_{\mathbbm K \in \mathcal M} a^{\mathbbm K}}{(a^{\mathbbm K})_{\mathbbm K \in \mathcal M}}
    \cdot \prod_{\mathbb K \in \mathcal M} (p_{\mathbb K}^r)^{a_{\mathbbm K}}
  \end{equation}
  with 
  \begin{align*}
    p_{\mathbb K}^v = e^{-\rho \lambda(\mathbb T^v_\mathbbm K)}
    \prod\limits_{ \mathbb S \in \mathbbm T^v \setminus \mathbbm T^v_{\mathbbm K}} 
    q^\ast(\mathbb S),
  \end{align*}
  the probability that a spacer gained at $v\in\mathbb V$ is
  present in $\mathbb K$ but absent in $\mathbb L\setminus \mathbb K$.
  In addition, let
  \begin{align*}
    C^{\mathbbm K} \stackrel d = \text{Poi}\left(    \frac\theta\rho \sum_{v \in (r; v_{\mathbb K}]\cap \mathbb V} 
    (1-e^{-\rho d(v,v\downarrow)})
    p_{\mathbb K}^v\right)
  \end{align*}
  be independent. Then,
  $$ (F_0^{\mathbbm K})_{\emptyset \subsetneq \mathbbm K\subsetneq \mathbbm L}, 
  (F_1^{\mathbbm K})_{\emptyset \subsetneq \mathbbm K\subsetneq
    \mathbbm L}, (F_2^{\mathbbm K})_{\emptyset \subsetneq \mathbbm
    K\subsetneq \mathbbm L},\dots$$ are independent with
  \begin{align*}
    (F_0^{\mathbbm K})_{\emptyset \subsetneq \mathbbm K\subsetneq \mathbbm L} 
    & \stackrel d = (C^{\mathbbm K} + A^{\mathbbm K})_{\emptyset \subsetneq \mathbbm K\subsetneq \mathbbm L},
    \\ 
    (F_i^{\mathbbm K})_{\emptyset \subsetneq \mathbbm K\subsetneq \mathbbm L} 
    & \stackrel d = (A_i^{\mathbbm K})_{\emptyset \subsetneq \mathbbm K\subsetneq \mathbbm L}, \qquad i=1,2,...
  \end{align*}
\end{theorem}

\begin{proof} 
 We use the same reasoning as in the proof of Theorem \ref{T:2}.
 Moving along the old spacers in $S_r$ that are present in the root of $\mathbbm T$, we have to distinguish $2^n-1$ cases, one for each $\mathbbm K \in \mathcal M$.
 The number $a^{\mathbbm K}$ corresponds to the number of spacers that have been lost in all $l \in \mathbbm L \setminus \mathbbm K$ and kept otherwise.
 The probability for a single spacer $s \in S_r$ to end up
 in the correct subset of $\mathbbm L$  is given by 
 \begin{equation*}
    P( s \in \bigcap\limits_{k \in \mathbbm K} S_k \setminus \bigcup\limits_{l \in \mathbbm L \setminus \mathbbm K} S_l | s \in S_r ) =
      p_{\mathbb K}^r
 \end{equation*}
 Just as in Theorem \ref{T:2} we are not interested in the spacers $s \notin \bigcup\limits_{l \in \mathbbm L} S_l$ that have been lost in all leaves and 
 \begin{equation*}
     P(s \in \bigcup\limits_{l \in \mathbbm L} S_l | s \in S_r) = p(r).
 \end{equation*}
 Finally we wait until the corresponding $2^n-1$-sided die falls on the side for an equal spacer with probability
 \begin{equation*}
     P(s \in \bigcap\limits_{l \in \mathbbm L} S_l | s \in S_r) = e^{-\rho \lambda(\mathbbm L)}.
 \end{equation*}
 To set everything together we count
 all possible combinations that end up in $(A^{\mathbbm K} = a^{\mathbbm K})_{\mathbbm K \in \mathcal M}$.
 This gives the multinomial coefficent in equation \eqref{eq:general} and we are done.
 
 For new spacers that have been gained along $\mathbbm T$ note that the number of 
 new spacers found in $\mathbbm K$ and $\mathbbm K' \neq \mathbbm K$ are independent Poisson random variables.
 Hence, we can first calculate the expected number of new spacers in $S_v$ that are gained between the two ancestral nodes $v_{\downarrow}$ and $v$ of $\mathbbm K$. 
 The probability that a spacer arises between $v_\downarrow$ and $v$,
  and is still present in $v$, is
  \begin{align*}
    \int_0^{d(v,v_\downarrow)} \theta e^{-\rho t}  dt = \frac\theta\rho (1-e^{\rho d(v,v_\downarrow)})
  \end{align*}

 To complete the result sum over all $v \in (r,v_{\mathbbm K}] \cap \mathbb V$ and multiply the summand with the probabilty 
 \begin{equation*}
  P( s \in \bigcap\limits_{k \in \mathbbm K} S_k \setminus \bigcup\limits_{l \in \mathbbm L \setminus \mathbbm K} S_l | s \in S_v )
  = p_{\mathbb K}^v
 \end{equation*}

 \end{proof}

\section{Estimating parameters} 
\label{S:4}
Our main results, Theorems~\ref{T:1} and~\ref{T:2}, are directly
applicable for maximum likelihood estimation of gain and loss
parameters. However, some care must be taken in order to produce
reliable results. In practice, several difficulties arise for a direct
application:
\begin{enumerate}
\item There is only a finite number of spacers:\\
  In our model, we assume that the observed equal spacers $i=1,...,m$
  are a finite subset of an infinite number of joint spacers. Using
  (simulated or real) data, the number of spacers is limited.
  In particular, under the ordered independent loss model we can compute 
  the likelihood restricted to the first $m$ equal spacers common to everyone.
  In contrast, in reality the number of equal spacers is a poisson distributed random variable $M$.
  The model could account for this fact by assuming a finite number of spacers in the ancestral species, but 
  in this case we would lose the regenerative property.
\item The phylogeny must be given in advance:\\
  While \cite{Kupczok2013} estimate the phylogeny (precisely, the time
  to the MRCA in a sample of size two) together with the spacer
  insertion/deletion ratio, we assume that the true phylogeny of
  the spacer arrays is known. In practice this phylogeny needs to be
  inferred from data, e.g.\ from a multiple alignment of the cas-genes
  or of conserved genomic regions.
\item There are more than three leaves:\\
  Theorems~\ref{T:1} and~\ref{T:2} give likelihoods only for trees
  with two and three leaves, respectively. If one wants to study
  larger samples, the corresponding full likelihood is computationally
  expensive.
  In order to directly
  apply our results, we will only use sample sizes of at most three in
  the sequel. 
\end{enumerate}

Often, insertion and loss rates can only be estimated relative to each
other, e.g. by estimating the ratio of both rates, see e.g.\
\cite{BaumdickerPfaffelhuberHess2012}. 
In contrast, for the ordered CRISPR spacer arrays, we can estimate the
spacer loss rate $\rho$ without any information about the insertion
rate $\theta$, if we restrict ourselves to spacers older than the
first common spacer; see Theorem \ref{T:1} and Theorem \ref{T:2}.
Having estimated $\rho$, the gene gain rate can e.g.\ be estimated by
observing that the average number of spacers is $\theta/\rho$ in an
equilibrium situation; see Remark~\ref{rem:more}.

In the following, we focus on the first step of estimating $\rho$ from
the spacer arrays $S_{\mathfrak 1}, S_{\mathfrak 2}$ in the case of a
sample of size~2 and from
$S_{\mathfrak 1}, S_{\mathfrak 2}, S_{\mathfrak 3}$ in the case of a
sample of size~3.

\subsection{Estimating $\rho$ based on samples of size two}
\noindent
Turning to the statistical problem of estimating $\rho$, we will use
the notion of an exponential family. Recall that a family of
distributions $\mathbbm P_\rho$ is called a $k$-parameter exponential family,
if, for some $\mathbb N^m$-valued $X \sim \mathbbm P_\rho$, it has the
form
$$ \mathbbm P_\rho(X = x) = h(x) \exp(\sum\limits_{i = 1}^k  c_i(\rho) t_i(x) - d(\rho)),$$
for some functions $h$, $c = (c_1,...,c_k), t = (t_1,...,t_k)$ and
$d$. Notice that $t$ is called the (minimal) sufficient statistics for
$(\mathbbm P_\rho)_\rho$, i.e.\ maximum likelihood inference about
$\rho$ depends on $x$ only through $t(x)$.

Assume we have two spacer arrays. To estimate the spacer
deletion rate using a maximum likelihood approach rate we define:

\begin{definition}[number of (un)equal spacers in samples of size
  $n=2$]\label{alg:1}
  Let $S_{\mathfrak 1},S_{\mathfrak 2}$ 
  be given. Calculate
  $$ M := |S_{\mathfrak 1} \cap S_{\mathfrak 2}|,$$
  the total number of equal spacers in $\mathfrak 1$ and $\mathfrak 2$
  and
  $$D = \sum_{i=2}^M (V_i-V_{i-1}) + (W_i - W_{i-1}) = V_M - V_1 + W_M-W_1 $$
  with
$
    V_i  := E_{i}^{\mathfrak 1, \{\mathfrak 1, \mathfrak 2\}} \text{and }
    W_i := E_{i}^{\mathfrak 2, \{\mathfrak 1, \mathfrak 2\}}.
$

\end{definition}

\begin{theorem}[Sufficient statistics and maximum likelihood estimator for $\rho$\label{T:suff}]
  \quad \\ a) Let $m\geq 2$. The joint distribution of
  $\Delta V_i := V_i-V_{i-1}, \Delta W_i := W_i - W_{i-1}$,
  $i=2,\dots,m$ only depends on $\rho$ and is a single-parameter
  exponential family with
  $$ t(\Delta v_2,\dots,\Delta v_m, \Delta w_2,\dots,\Delta w_m) = 
  \sum_{i=2}^m \Delta v_i + \Delta w_i = (v_m-v_1) + (w_m-w_1).$$
  \\
  b)
  Let $D$ and $M$ be fixed to $d,m \geq 0$ respectively.
  Then, the maximum likelihood estimator of $\rho$ is given by 
  \begin{equation}\label{eq:totalprob2}
    \rho^\ast 
    = \mathop{argmax}\limits_{\rho} \left[(M-1)\log\bigg(\frac{e^{-\rho T}}{2-e^{-\rho T}} 
    \bigg) + D \log\bigg(\frac{1-e^{-\rho T}}{2 - e^{-\rho T}}
    \bigg)\right]
  \end{equation}
  and can be computed
explicitly:
  \begin{align}
    \label{eq:rhoast}
    \rho^\ast = -\frac{\log p^\ast}{T} \qquad \text{with}\qquad 
    p^\ast = \bigg(1 + \frac{d}{2(m-1)}\bigg)^{-1}
  \end{align}
\end{theorem}

\begin{proof}
  a)
  Using Theorem~\ref{T:1}, we write
  $\underline{\Delta v} := (\Delta v_2, \dots, \Delta v_m)$
 and 
  $\underline{\Delta w} := (\Delta w_2, \dots, \Delta w_m)$, such that
  \begin{align*}
    \mathbbm P( & \underline{\Delta V} = \underline{\Delta v}, \underline{\Delta W} = \underline{\Delta w})
    \\ & = 
         \bigg(\prod_{i=2}^m \binom{\Delta v_i + \Delta w_i}{\Delta v_i}\bigg)
         \cdot \bigg(\frac{e^{-\rho T}}{2-e^{-\rho T}} 
         \bigg)^{m-1}
         \cdot \bigg(\frac{1-e^{-\rho T}}{2 - e^{-\rho T}}
         \bigg)^{\sum_{i=2}^m \Delta v_i + \Delta w_i}\\
                & = h(\underline{\Delta v}, \underline{\Delta w}) \exp(c(\rho) 
                  t(\underline{\Delta v}, \underline{\Delta w}) - d(\rho))
  \end{align*}
  with
  \begin{align*}
    h(\underline{\Delta v}, \underline{\Delta w}) & := \prod_{i=2}^m \binom{\Delta v_i + \Delta w_i}{\Delta v_i},\\
    c(\rho) & = \log\bigg(\frac{1-e^{-\rho T}}{2 - e^{-\rho T}}\bigg),\\
    d(\rho) & = - (m-1) \log \bigg(\frac{e^{-\rho T}}{2-e^{-\rho T}} 
         \bigg)
  \end{align*}
  and $t$ from above.
  
  ~
  
  \noindent
  b)
    Formula \eqref{eq:totalprob2} follows directly from
  Corollary \ref{cor:1}. Further, we compute for $p = e^{-\rho T}$
  \begin{align*}
    \frac{d}{dp} & (M-1)\log\Big(\frac{p}{2-p}\Big) + D\log\Big(\frac{1-p}{2-p}\Big)
                   = (M-1) \Big( \frac 1p + \frac{1}{2-p} \Big) + D \Big(\frac{1}{2-p} - \frac{1}{1-p}\Big)
    \\ & = \frac{1}{2-p}\Big(\frac {2(M-1)}{p} -  \frac{D}{1-p}\Big) 
         =  \frac{1}{p(1-p)(2-p)}\bigg(2(M-1)(1-p) - p D\bigg).
  \end{align*}
  Hence, the maximum likelihood estimator is
  $$ p^\ast = \frac{2(M-1)}{2(M-1) + D} = \bigg(1 + \frac{D}{2(M-1)}\bigg)^{-1}.$$
\end{proof}

\begin{remark}[Assuming independence leads to the same estimator for
  $n=2$\label{rem:98}]
  Recall from Remark \ref{rem:geo1} that $V_i-V_{i-1}$ (as well as
  $W_i-W_{i-1}$) is geometrically distributed with success parameter
  $p/(2-p)$ for $p=e^{-\rho T}$. Moreover, $V_M-V_1$ (as well
  as $W_M-W_1$) has a negative binomial distribution with parameters
  $M-1$ and $p$. Assuming that $V_M-V_1$ and $W_M-W_1$ are independent
  (which is clearly false in our model), $D$ has a negative binomial
  distribution with parameters $2(M-1)$ and $p$. A simple calculation
  shows that the maximum likelihood estimator of $p$ is again given by
  $p^\ast$ from \eqref{eq:rhoast}. Therefore, the maximum likelihood
  estimator for $\rho$ in the case of a sample of size~2 can also be
  obtained by assuming independence of $V_M-V_1$ and $W_M-W_1$.
\end{remark}

\subsection{Estimating $\rho$ based on samples of size three}
\noindent
For $n=3$, we again obtain an exponential family, since
$((F_i^{\mathbbm K})_{\emptyset \subsetneq \mathbbm K\subsetneq
  \mathbbm L})_{i=1,2,...}$
are independent and identically distributed according to
\eqref{eq:ABCDEF}.

\begin{theorem}[Sufficient statistics for $\rho$\label{T:suff2}]
  Let $m\geq 1$. The joint distribution of
  $((F_i^{\mathbbm K})_{\emptyset \subsetneq \mathbbm K\subsetneq
    \mathbbm L})_{i=1,...,m}$ is a 4-parameter exponential family with
  \begin{align*}
    t_1 &= \sum_{i=1}^m F_i^{\{\mathfrak 1\}} + F_i^{\{\mathfrak 2\}},\quad
    t_2 = \sum_{i=1}^m F_i^{\{\mathfrak 3\}},\quad
    t_3 = \sum_{i=1}^m F_i^{\{\mathfrak 1, \mathfrak 2\}},\quad
    t_4 = \sum_{i=1}^m F_i^{\{\mathfrak 1, \mathfrak 3\}} + F_i^{\{\mathfrak 2, \mathfrak 3\}},
  \end{align*}
  \begin{align*}
    c_1(\rho) & = \log\Big(\frac{(1-e^{-\rho T})(1-e^{-\rho T'})}{3 - e^{-\rho T'} - 
                e^{-\rho T}(2-e^{-\rho T'})}\Big), 
    && 
       c_2(\rho)  = \log\Big(\frac{1-2e^{-\rho T} + e^{-\rho (T+T')})}{3 - e^{-\rho T'} - 
       e^{-\rho T}(2-e^{-\rho T'})}\Big), \\
    c_3(\rho) & = \log\Big(\frac{e^{-\rho T'}(1- e^{-\rho T})}{3 - e^{-\rho T'} - 
                e^{-\rho T}(2-e^{-\rho T'})}\Big),
    &&
       c_4(\rho)  = \log\Big(\frac{e^{-\rho T}(1- e^{-\rho T'})}{3 - e^{-\rho T'} - 
       e^{-\rho T}(2-e^{-\rho T'})}\Big),
  \end{align*}
  \begin{align*}
    d(\rho) & = -m \log\Big(\frac{e^{-\rho (T+T')}}{3 - e^{-\rho T'} - e^{-\rho T}(2-e^{-\rho T'})}\Big).
  \end{align*}
\end{theorem}

Analogous to the case of two spacer arrays we can use the results of
Theorem \ref{T:2} and \ref{T:suff2} to build an estimator based on three spacer arrays.

\begin{definition}[loss rate estimation in samples of size
  $n=3$]\label{alg:2}
  Let $S_{\mathfrak 1}, S_{\mathfrak 2}, S_{\mathfrak 3}$ be given,
  and denote the time to the MRCA from $\mathfrak 1$ and $\mathfrak 3$
  ($\mathfrak 1$ and $\mathfrak 2$) by $T$ ($T'$) with $T\geq T'$.
  Calculate
  $$ M:=|S_{\mathfrak 1} \cap S_{\mathfrak 2} \cap S_{\mathfrak 3}|,$$
  the total number of equal spacers in $\mathfrak 1, \mathfrak 2$ and
  $\mathfrak 3$ and
  \begin{align*}
    D_1 & := \sum_{i=1}^M F_i^{\{\mathfrak 1\}} + F_i^{\{\mathfrak 2\}}, \qquad  
    & D_2 & := \sum_{i=1}^M F_i^{\{\mathfrak 3\}}, 
    \\ D_3 & := \sum_{i=1}^M F_i^{\{\mathfrak 1, \mathfrak 2\}}, \qquad  
    & D_4 & := \sum_{i=1}^M F_i^{\{\mathfrak 1, \mathfrak 3\}} + F_i^{\{\mathfrak 2, \mathfrak 3\}}.
  \end{align*}
  then, set (compare with \eqref{eq:ABCDEF})
  \begin{align*}
    \rho^\ast & = \mathop{argmax}\limits_{\rho} 
                \Big(- (M-1) \rho (T+T') 
                + D_1 \log\big((1-e^{-\rho
                T})(1-e^{-\rho T'}\big)
    \\ & \qquad \qquad \qquad \qquad \qquad 
         + D_2 \log\big(1-2e^{-\rho T} + e^{-\rho 
         (T+T')})\big)
    \\ & \qquad \qquad \qquad \qquad \qquad 
         +D_3\log\big(e^{-\rho T'}(1-
         e^{-\rho T})\big)
    \\ & \qquad \qquad \qquad \qquad \qquad 
         + D_4 \log\big(e^{-\rho
         T}(1- e^{-\rho T'})\big)
    \\ & \qquad \qquad \qquad \qquad \qquad 
         - (M-1+D_1+D_2+D_3+D_4) \log(3 - e^{-\rho T'} - e^{-\rho T}(2-e^{-\rho T'})).
  \end{align*}
\end{definition}

\noindent
Maximizing the right hand side of the last expression is hardly done
explicitly. We therefore rely on numerical optimization in our
applications.

\subsection{Comparison of estimation methods}
We used simulated data in order to compare our estimates based on
\begin{enumerate}
\item the exact likelihoods based on equal spacers
    from Definition~\ref{alg:1} for $n=2$ and Definition~\ref{alg:2}
    for $n=3$, (darkgray boxes in Figure~\ref{fig1}) and
\item the software {\em panicmage}, which is based
    on \cite{BaumdickerPfaffelhuberHess2012} and uses the spacer
    frequency spectrum in order to simultaneously estimate $\theta$
    and $\rho$ (lightgray boxes in Figure~\ref{fig1}).
\end{enumerate}
In the simulations, we consider gains and losses in equilibrium, i.e.\
the root of the tree has a Poisson number of spacers with mean
$\theta/\rho$. For the trees relating the leaves, we take the
coalescent as a null model, i.e.\ pairs of lines all coalesce
independent from each other at constant rate $1$ independently of $\rho$; see e.g.\
\cite{Wakeley2008}. This has the advantage that we can use results
from \cite{BaumdickerPfaffelhuberHess2012} in order to interpret our
results. For example, we know that in a sample of size $n=2$ (and
$n=3$), the number of spacers common to both arrays (previously called
$M$) has expectation $\frac{\theta}{\rho(1 + 2\rho)}$ (and
$\frac{\theta}{\rho(1 + 2\rho)(2+2\rho)}$).
We simulated $1000$ replicates. For each replicate a new coalescent tree was drawn.
Throughout all simulations we set $\theta = 100\rho$, such that the number of spacers 
in an array is Poisson distributed with mean $100$.

 \begin{figure}
\phantom{.}\hfill \subfloat[$n = 2$]{
  \begin{overpic}[width=0.45\textwidth,trim=1cm 1cm 1cm 1cm,clip]%
  {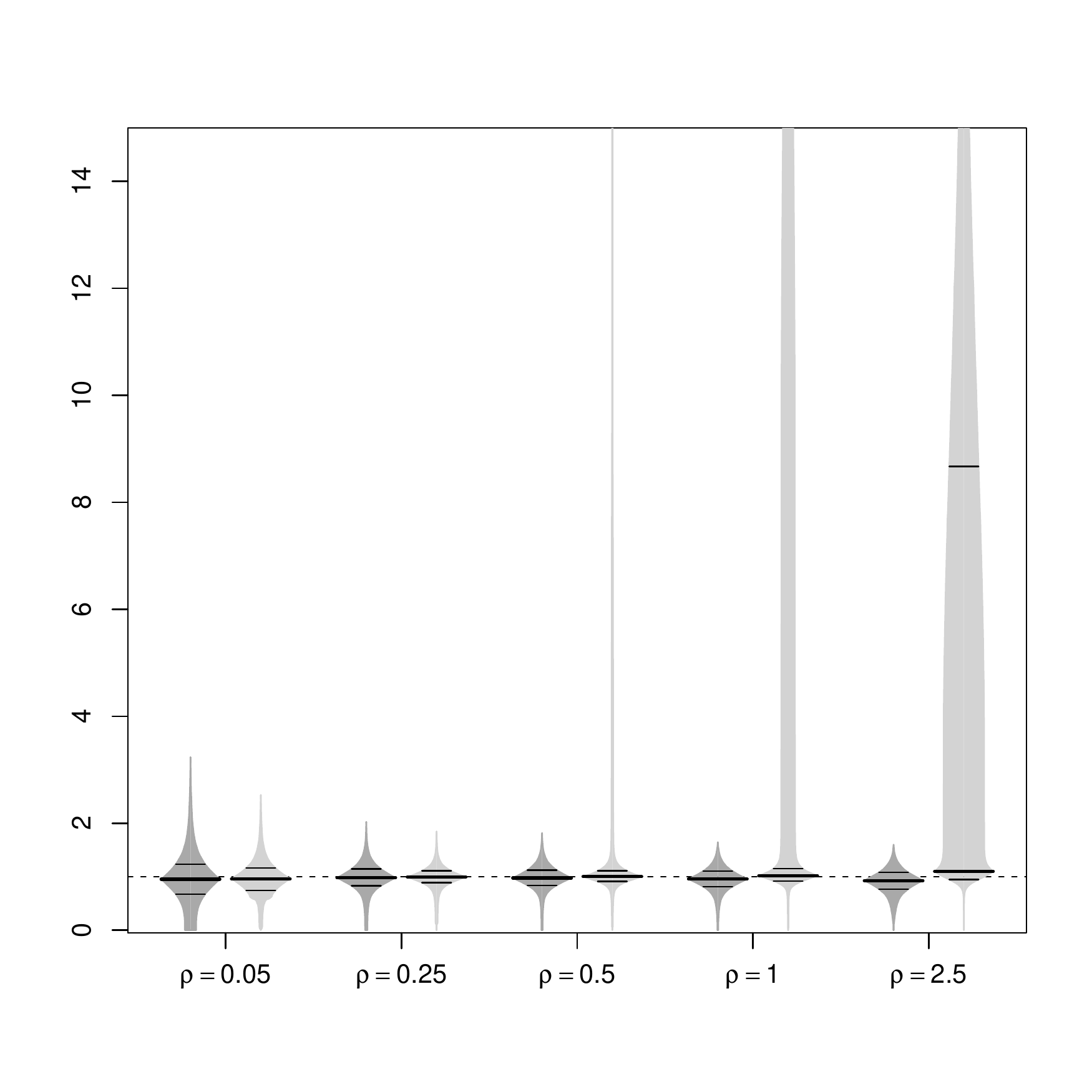}
  \put(-10,45){\rotatebox{90}{$\hat\rho/\rho$}} 
  \end{overpic}
}
\subfloat[$n = 3$]{
    \includegraphics[width=0.45\textwidth,trim=1cm 1cm 1cm 1cm,clip]{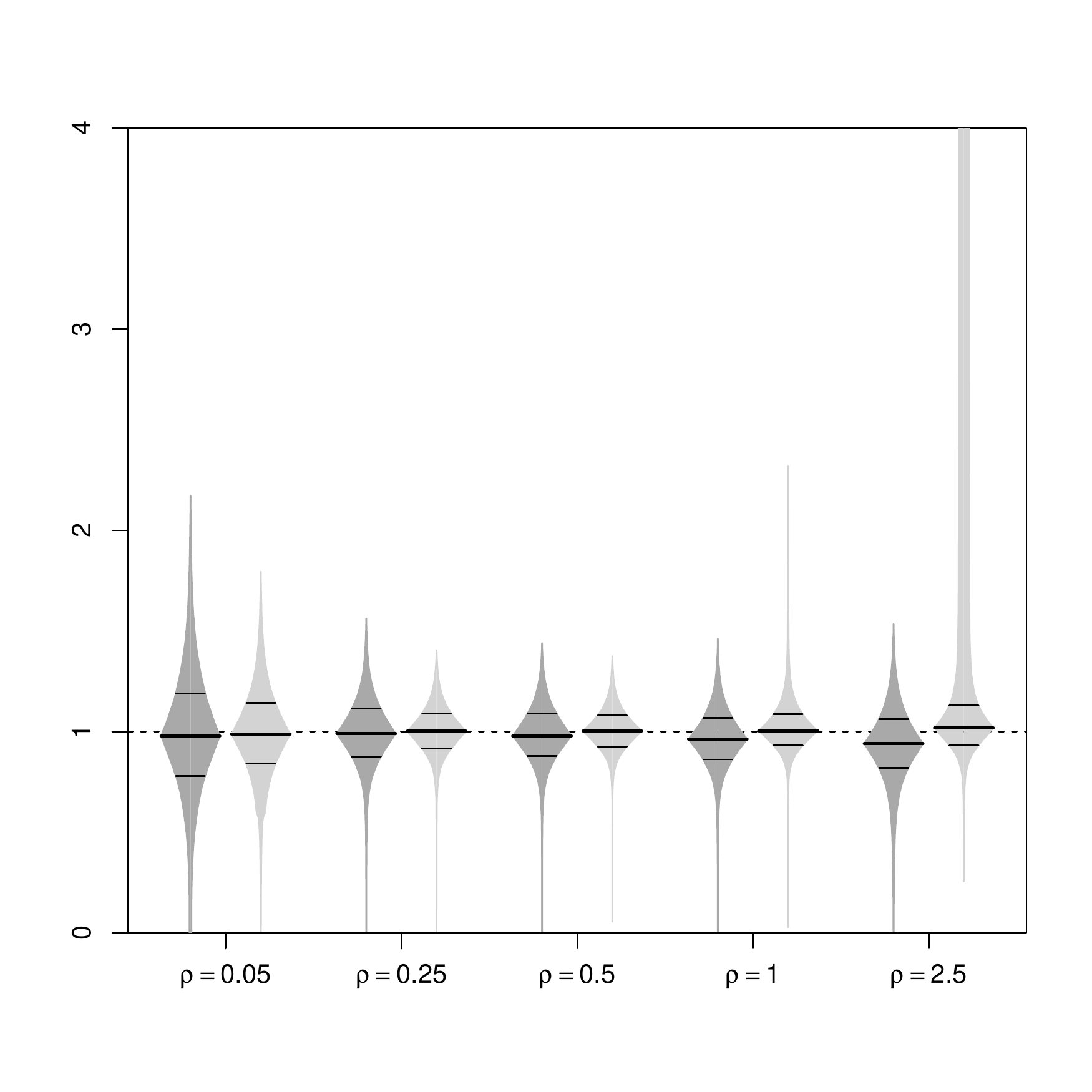}    
}
\caption{\label{fig1}Box-Percentile plots \citep{Esty2003} of estimated
  deletion rates $\hat\rho$ divided by true parameter $\rho$. The
  estimates are based on two-way comparisons (left boxes, darkgray) and the
  spacer frequency spectrum (right boxes, lightgray) for different deletion
  rates $\rho$. We use $\theta = 100\rho$ in all cases. (a) Estimates
  based on a sample of size $n=2$; see Definition~\ref{alg:1}. (b)
  Estimates based on a sample of size $n=3$; see
  Definition~\ref{alg:2}}
\end{figure}

~

As Figure \ref{fig1} shows, both estimation procedures give accurate
results (in the sense that the estimator $\hat\rho$ is close to
$\rho$, or $\hat\rho/\rho\approx 1$ in many cases), but there are some
differences. Since the equal spacer likelihood methods are based on
extensions of geometric distributions, the resulting estimators are
(slightly) biased. In contrast, the likelihood based on the spacer
frequency spectrum can be calculated from independent Poisson random
variables \citep{BaumdickerHessPfaffelhuber2010} and the corresponding estimator is thus unbiased.
Concerning the variance of the estimators, the equal spacer likelihood
methods give a slightly larger variance of the estimators for small
$\rho$ (for $n=2, \rho\leq0.25$, Fig.~\ref{fig1}(a); and for $n=3, \rho\leq0.5$, Fig.~\ref{fig1}(b)).
This might be due to the fact that this estimation procedure
only takes spacers after the first common spacer into
account. However, after the first equal spacer, most spacers are
equal for low $\rho$, such that there is less signal in the data, making estimation of $\rho$ more difficult.
Note that for larger loss rates (for $n=2, \rho\geq0.5$, Fig.~\ref{fig1}(a); and for $n=3, \rho\geq1$, Fig.~\ref{fig1}(b)), 
the estimate from {\em panicmage} has a much higher
variance. This might be due to the fact that {\em panicmage} cannot
estimate $\theta$ and $\rho$ separately and the likelihood surface is
flat in one direction in the $\theta/\rho$-plane.

\section{Discussion} 
\label{S:5}

The ordered independent loss model for CRISPR spacer evolution paves
the way for various statistical applications based on equal
spacers. However, when comparing the estimates for the spacer deletion rate with the unordered model in small sample sizes and limited deletion rates, the extra information 
of the order of spacers does not seem to produce more reliable results; see
Figure~\ref{fig1}. Nevertheless, this extra information will prove to be useful
in further applications. In particular, since spacers are ordered by time of their insertion in each
array, we can see which spacers in the arrays were present in the nodes of the phylogeny. If a
spacer is present in a specific node in the phylogeny, all subsequent
spacers were inserted before and hence must also be present within this
node. This information is only available in the ordered model and will help to understand
phenomena such as non-constant insertion rates or horizontal transfer.

The ordered independent loss model, as presented above, assumes constant insertion and deletion rates.
The separate estimation of spacer insertion and deletion rates allows to investigate each mechanism independently.
Besides selective effects, 
the insertion of spacers is more strongly dependent on the environment of the population than the random deletion of spacers is.
The insertion might hence be time-dependent \citep{Hynes2016applicationsofCRISPRinnature}.
In contrast, it is more reasonable to assume that the deletion of a neutral spacer occurs at constant rate.
To test this hypothesis in future statistical work, we therefore aim at methods based on equal spacer distances.
This will allow us to infer deletion rate variation between different branches in the given phylogeny,
different CRISPR-Cas systems or positions in the spacer array.

To apply the presented approach to genmic data, the CRISPR spacer arrays and a reconstructed phylogeny are necessary.
In addition the equal spacers among all the CRISPR arrays need to be identified.
There are several CRISPR prediction tools that can be used to scan for spacer arrays in genomic data \citep{Bland2007,Edgar2007,Grissa2007finder}.
For archaeal and bacterial reference genomes predicted spacers are available online \citep{Grissa2007database}.
While a typical array has less than 20 spacers, there are also arrays with several hundred spacers.

The ordered independent loss model will be useful to study the evolutionary ecology of CRISPR-Cas systems.
For an extensive review of the evolution and ecology of CRISPR we refer the reader to \cite{Westra2016}.
Here we will highlight the role of spacer deletions in CRISPR-enhanced gene drives and metagenomic studies concerning CRISPR.
Gene drive is the dependency of the offspring distribution of an individual on its genetic background. As such, this concept in concert with CRISPR
can be used for the artificial alterations of natural populations \citep{Burt2003,Hynes2016applicationsofCRISPRinnature}.
However,  \cite{Unckless2017} argue that in natural populations resistance to a gene drive will likely occur.
As a counter measure, different improvements of the CRISPR gene drive have been proposed \citep{Esvelt2014}.
One suggestion is to use a CRISPR gene drive with more than one spacer, such that resistance can not evolve so easily.
The effect of losing spacers in such a gene drive should be considered carefully before using any designed CRISPR system in natural populations.
The same is true for the idea to "vaccinate" bacteria possessing CRISPR via defective phage and plasmid sequences \citep{Hynes2014,Hynes2016programmingCRISPR}.
Since any artificial CRISPR spacer insertion is transmitted to the offspring we should assess the fixation probabilities and extinction times
of vaccinated bacterial strains before releasing them into natural environments.

To better understand the dynamics of CRISPR and the corresponding phages in natural populations 
metagenomic studies have been used.
Using the sequenced CRISPR spacer arrays, the ancestral environment of the population can be inferred \citep{Sun2015}.
In addition, different lineages can be distinguished, even for highly related bacterial strains \citep{Kunin2008}.
Since the ordered independent loss model is based on a genealogical tree connecting the individuals, it cannot directly be used for metagenomic data.
Nonetheless, the model provides a theoretical framework to have a more detailed look at the evolutionary pattern in the spacer array.
In particular, the often observed conservation of old spacers at the trailer-end of the array is of interest \citep{Weinberger2012}.
In the ordered independent loss model conservation of spacers correspond to small equal spacer distances, i.e. a low spacer deletion rate $\rho$.
For large loss rates the number of spacers between equal spacers rises until finally all observed spacers differ.
In a non-neutral setting with selection simulations of CRISPR arrays and the viral population show that the distances between equal spacers still increase for high spacer deletion rates.
To which amount the CRISPR trailer-end is conserved due to selective forces compared to the effect of a low deletion rate remains a future challenge.

\subsubsection*{Acknowledgments}
We thank Omer Alkhnbashi and Rolf Backofen for fruitful discussion
and two anonymous referees for their constructive comments and hints.
We thank an anonymous referee to encourage us to give our most general result, Theorem~3.
This research was supported by the DFG through the priority program
SPP1590. In particular, FB is funded in parts through Pf672/9-1.

\section*{Bibliography}

\bibliography{crispr}

\end{document}